\documentclass[a4paper]{amsart}
\usepackage[utf8]{inputenc}
\usepackage{amssymb}
\usepackage{amsthm}
\usepackage{amsmath}
\usepackage{mathtools}
\usepackage{enumitem}
\usepackage{nicefrac}
\usepackage{xcolor}
\usepackage{bm}
\usepackage{hyperref}

\title[Independence of premise for IZF]{Independence of premise for intuitionistic Zermelo-Fraenkel set theory}
\author{Patrick Uftring}

\address{University of the Bundeswehr Munich, Department of Computer Science, Werner-Heisenberg-Weg 39, 85579 Neubiberg, Germany}
\email{patrick.uftring@unibw.de}

\keywords{intuitionistic set theory, independence of premise, Friedman-Dragalin translation}
\subjclass[2020]{03F03; 03F50, 03F55}

\newcommand{\czf}{\textsf{CZF}}
\newcommand{\izf}{\textsf{IZF}}
\newcommand{\zfc}{\textsf{ZFC}}
\newcommand{\period}{\text{.}}
\newcommand{\comma}{\text{,}}

\newcommand{\haomega}{\textsf{HA}^{\omega}}
\newcommand{\ha}{\textsf{HA}}
\newcommand{\mr}{\mathrel{\textsf{mr}}}

\newcommand{\markov}{\textsf{MP}}

\begin{document}
	
	\theoremstyle{plain}
	\newtheorem{lemma}{Lemma}
	\numberwithin{lemma}{section}
	\newtheorem{proposition}[lemma]{Proposition}
	\newtheorem{corollary}[lemma]{Corollary}
	\newtheorem{theorem}[lemma]{Theorem}
	\newtheorem{conjecture}[lemma]{Conjecture}
	\newtheorem*{theorem*}{Theorem}
	
	\theoremstyle{definition}
	\newtheorem{definition}[lemma]{Definition}
	\newtheorem{question}[lemma]{Question}
	\newtheorem{convention}[lemma]{Convention}
	\newtheorem{problem}[lemma]{Problem}
	
	\theoremstyle{remark}
	\newtheorem{remark}[lemma]{Remark}
	
	\begin{abstract}
	We answer the following question by Frittaion, Nemoto, and Rathjen positively: ``Is [independence of premise] an admissible rule of \czf{} or any other familiar constructive/intuitionistic set theory~${T}$?'' To be more precise, we show that whenever \izf{} (or \czf{} with full separation) derives a statement of the form $$\lnot\psi \to \exists y\ \varphi(y)$$ where~${y}$ does not occur in~${\psi}$, then \izf{} (or \czf{} with full separation) \emph{also} derives $$\exists y\ (\lnot\psi \to \varphi(y))\text{.}$$ Our proof method is a variant of the famous Friedman-Dragalin~${A}$-translation adapted to the context of set theory in a hereditary manner.
\end{abstract}
\maketitle{}
	
\section{Introduction}\label{sec:intro}

In order to motivate the principle of independence of premise and explain its significance, let us consider Kreisel's modified realizability (cf.~\cite{Kreisel59, Kreisel62}): one of the most famous tools in theoretical proof extraction. Given a formula~${\varphi}$ that is provable in Heyting arithmetic in all finite types~${\haomega}$ (but also further axioms, see \cite[Theorem~3.4.5, Lemma~3.4.7]{Troelstra} and \cite[Theorem~5.8]{Kohlenbach}), we extract a term~${t}$ in G\"{o}del's~${T}$ that (provably) realizes the validity of~${\varphi}$. If our formula~${\varphi}$ is an existence statement~${\exists x\ \psi(x)}$, then~${t}$ comes with an explicit witness~${c}$ for~${x}$ and a realizer for~${\psi(c)}$. If our formula~${\varphi}$ is an implication~${\psi \to \chi}$, then~${t}$ is a functional that maps realizers of~${\psi}$ to realizers of~${\chi}$.

Over~${\haomega}$, the strength of~${\varphi}$ and the statement that~${\varphi}$ has a realizer may be distinct, even incomparable.\footnote{Let~${\varphi}$ be the statement that full choice and the consistency of~${\zfc{}}$ are equivalent. Since full choice has a trivial realizer and consistency is a~${\Pi^0_1}$-statement, the existence of a realizer for~${\varphi}$ is equivalent to the consistency of~${\zfc{}}$ alone. Over~${\haomega}$, this statement and~${\varphi}$ are incomparable.} Still, it is possible to enrich~${\haomega}$ with further axioms to ensure that both~${\varphi}$ and ``$\varphi$ has a realizer'' are equivalent (cf.~\cite[Theorem~3.4.8]{Troelstra} and \cite[Theorem~5.12]{Kohlenbach}). For the purposes of this article, our focus lies on the case involving implications. Let us, again, assume that~${\varphi}$ is an implication of the form~${\psi \to \chi}$. Our goal is to derive that~${\varphi}$ has a realizer. Inductively, we can assume that our claim is already satisfied for both~${\psi}$ and~${\chi}$. Thus, we are in the following situation:
\begin{equation*}
	\exists x\ x \mr{} \psi \to \exists y\ y \mr{} \chi\period
\end{equation*}
Here,~${x \mr{} \psi}$ and~${y \mr{} \chi}$ express that~${x}$ and~${y}$ are realizers for~${\psi}$ and~${\chi}$, respectively. In order to construct a realizer for~${\psi \to \chi}$, we aim to extract both~${\exists}$-quantifiers and then apply choice to build a functional out of the ensuing~${\forall\exists}$-statement. Our first step is simple: Prenexing the existential quantifier of our premise is intuitionistically justified:
\begin{equation*}
	\forall x.\ x \mr{} \psi \to \exists y\ y \mr{} \chi.
\end{equation*}
The~${\exists}$-quantifier in the conclusion, however, is problematic: In plain intuitionistic logic, but also~${\haomega}$, we are not allowed to move this quantifier over our premise.\footnote{Choose~${\chi := \exists x\ T(z, z, x)}$ and~${\psi := \lnot\lnot\chi}$ for Kleene's~${T}$ predicate. Conservativity (cf.~\cite[Corollary~4.2]{Goodman76}) yields a statement in~${\ha + \markov}$ that is not Kleene realizable (cf.~\cite{Kleene45}).} What we need is \emph{independence of premise} for implications with~${\exists}$-free premises or equivalently (as the presence of choice yields equivalences to negated formulas, cf.~\cite[Corollary~5.16]{Kohlenbach}), the following principle:
\begin{equation}\tag{$*$}
	(\lnot \psi \to \exists x\ \chi(x)) \to \exists x (\lnot \psi \to \chi(x))
\end{equation}
where~${x}$ may not occur freely in~${\psi}$.
This axiom, although not derivable in~${\haomega}$, can be justified both with respect to the BHK-interpretation and modified realizability itself by the fact that negated formulas have empty realizers. Hence, a (total) functional interpreting the implication in the premise takes no input and immediately comes with a witnessing candidate~${c}$ for~${x}$ such that~${\chi(c)}$ holds as soon as~${\lnot\psi}$ is ensured. The candidate~${c}$ is fully determined \emph{before} the validity of our premise~${\lnot\psi}$ comes into play. Independence of premise is not only relevant in the context of modified realizability but also for G\"odel's Dialectica interpretation. See~\cite[Definition~8.4]{Kohlenbach} as well as \cite[p.~127]{Kohlenbach} for more information.

When axiom schemes such as ($*$) are not provable, it is often still possible to derive related \emph{rules}: While~${\haomega{}}$ does not prove Markov's principle, it admits Markov's rule (cf.~\cite[Lemma~3.3]{Friedman78}). While it does not prove the scheme of full choice, it admits a choice rule (cf.~\cite[Corollary~5.24]{Kohlenbach}). In order to demonstrate a similar phenomenon with respect to ($*$), let us make the formulation of our rule precise:
\begin{definition}\label{def:IP}
	A theory~${T}$ admits the rule of \emph{independence of premise} for negated premises if any derivation of
	\begin{align*}
		T &\vdash \lnot \psi \to \exists x\ \chi(x)
		\shortintertext{yields a proof}
		T &\vdash \exists x.\ \lnot \psi \to \chi(x)
	\end{align*}
	where~${x}$ may not occur freely in~${\psi}$.
\end{definition}
At this point, it is probably not too much of a surprise that~${\haomega}$ also admits this rule.
The admissibility of independence of premise and full choice \emph{as rules} can be verified by simple exercises in the context of modified realizability (with truth, cf.~\cite[Corollary~5.24]{Kohlenbach}). However, the situation gets much more complex if we move from arithmetic to (constructive) set theories such as~${\czf{}}$ or~${\izf{}}$.

In \cite{FNR23}, Frittaion, Nemoto, and Rathjen study a realizability interpretation of~${\czf{}}$ (and several extensions) based on work by Rathjen, who originally involved truth into an adaptation of a definition by McCarty to derive disjunction and numerical existence properties for~${\czf{}}$. (See \cite[p.~282]{Rathjen06} and \cite[pp.~1235--1236]{Rathjen05} for a more detailed historical context.) Using this interpretation, the authors derive several rules involving finite types. In particular, they work on an independence of premise rule and prove the following result:
\begin{theorem*}
	Let~${T}$ be one of the theories in \cite[Theorem~4.11]{FNR23}, which includes both~${\czf{}}$ and~${\izf{}}$. Assume
	\begin{align*}
		T &\vdash \lnot\psi \to \exists y^\sigma \varphi(y)
		\shortintertext{where~${y}$ does not occur in~${\psi}$ and~${\sigma}$ is a finite type. Then, also}
		T &\vdash \exists y.\ \lnot \psi \to y \in \sigma \land \varphi(y)
	\end{align*}
	holds.
\end{theorem*}
Notice that~${y}$ is separated from its type annotation, which stays behind: The extraction of~${y}$ is possible due to the realizability notation together with a trick involving a quantification over the contents of all functionals of type~${\sigma}$ (see the proof of \cite[Theorem~8.5]{FNR23} for details). However, the verification of~${y \in \sigma}$ depends on a member of the employed partial combinatory algebra, which is only guaranteed to be defined when~${\lnot\psi}$ can be assumed. The question whether~${y \in \sigma}$ can be extracted is formulated as \cite[Problem~1.4]{FNR23}.
The question that \emph{this} article aims to answer is the subsequent \cite[Problem~1.5]{FNR23}:

``Is [Definition~\ref{def:IP}] an admissible rule of~${\czf{}}$ or any other familiar constructive/intuitionistic set theory~${T}$?''

Although our methods (regrettably) do not seem to work for~${\czf{}}$ in particular, we prove the rule from Definition~\ref{def:IP} for both~${\czf{}}$ \emph{with full separation} and~${\izf{}}$ in Theorem~\ref{thm:main}. Thus, we answer \cite[Problem~1.5]{FNR23} positively using well-established constructive set theories.

Instead of going through the route of realizability, we use a completely different approach: the Friedman-Dragalin translation (cf.~\cite{Friedman78, Dragalin80}). Traditionally, this method was used to derive Markov's rule (cf.~\cite[Lemma~3.3]{Friedman78}): Given formulas~${\varphi}$ and~${P}$, let~${\varphi^P}$ denote the formula that results from~${\varphi}$ if each prime formula~${s =_0 t}$ and~${\bot}$ is replaced with~${(s =_0 t) \lor P}$ and~${P}$, respectively. (Here, it is important that negations are expressed using implications and~${\bot}$.) In an effort to derive Markov's rule, assume that the double negation of a simple existence statement is derivable:
\begin{equation*}
	\haomega \vdash \lnot\lnot \exists x \varphi(x)
\end{equation*}
where quantifiers in~${\varphi(x)}$ are bounded. Our goal is to remove the double negation. Since logical connectives and bounded quantifiers can be expressed by primitive recursive functionals, we find a term~${t}$ in G\"{o}del's~${T}$ such that (provably in~${\haomega}$)~${\varphi(x)}$ is satisfied if and only if~${tx =_0 0}$ holds. Now, we apply the Friedman-Dragalin translation to the resulting statement and consider~${(\lnot\lnot \exists x\ tx =_0 0)^{\exists x \varphi(x)}}$. By \cite[Lemma~3.2]{Friedman78}, this translation is also provable in~${\haomega}$. Hence, we have
\begin{equation*}
	\haomega \vdash (\exists x\ (tx =_0 0 \lor \exists x\ \varphi(x)) \to \exists x\ \varphi(x)) \to \exists x\ \varphi(x)\period
\end{equation*}
Miraculously, this results in a derivation of
\begin{equation*}
	\haomega \vdash \exists x\ \varphi(x)\period
\end{equation*}
Seemingly out of thin air, a classical proof has been turned into a constructive one. See \cite[Chapter~14]{Kohlenbach} for applications and further literature.

Although the independence of premise rule traditionally is derived using modified realizability with truth, it is also possible to prove it using the Friedman-Dragalin translation. As it turns out, this idea is nothing new and, in the context of Heyting arithmetic, already appears in Albert Visser's PhD thesis (cf.~\cite[Section~3.1]{VisserPhD}, also \cite[Theorem~5.5 of Chapter~3]{TvD88}). Similar to the situation with Markov's rule, we aim to attack the negation in our premise. By choosing the right translation, our goal is to eliminate the premise completely, which trivially pulls the existential quantifier out. In an effort to derive independence of premise, assume
\begin{equation*}
	\haomega \vdash \lnot \psi \to \exists y\ \varphi(y)\period
\end{equation*}
This time, we apply the Friedman-Dragalin translation via the formula~${\lnot\lnot\psi}$. Using the same result as before, this yields
\begin{equation*}
	\haomega \vdash (\psi^{\lnot\lnot\psi} \to \lnot\lnot\psi) \to \exists y\ \varphi(y)^{\lnot\lnot\psi}\period
\end{equation*}
Our focus lies on the premise. Assume~${\psi^{\lnot\lnot\psi}}$ and~${\lnot\psi}$. By the latter, all occurrences of~${\lnot\lnot\psi}$ in~${\psi^{\lnot\lnot\psi}}$ vanish. This results in our original formula~${\psi}$. Now, by a second application of~${\lnot\psi}$, we derive~${\bot}$. Hence, our premise~${\psi^{\lnot\lnot\psi} \to \lnot\lnot\psi}$ is just a tautology and we arrive at
\begin{equation*}
	\haomega \vdash \exists y\ \varphi(y)^{\lnot\lnot\psi}\period
\end{equation*}
Now that our existential quantifier has been extracted, we are almost done: We only need to re-introduce the premise~${\lnot\psi}$. Similar to before, this eliminates any occurrence of~${\lnot\lnot\psi}$ in the formula~${\varphi(y)^{\lnot\lnot\psi}}$. Consequently, we arrive at our claim
\begin{equation*}
	\haomega \vdash \exists y.\ \lnot\psi \to \varphi(y)\period
\end{equation*}
The obvious question is: Can this approach be transferred to a constructive set theory such as~${\czf{}}$? Unfortunately, it is not that simple. In order to demonstrate the issue, assume that our Friedman-Dragalin translation of a formula~${\varphi}$ to~${\varphi^P}$ (for~${\varphi}$ and~${P}$ in the language of set theory) directly replaces all prime formulas~${x = y}$,~${x \in y}$, and~${\bot}$ in~${\varphi}$ with~${(x = y) \lor P}$,~${(x \in y) \lor P}$, and~${P}$, respectively. Then, consider the translation of extensionality, which is usually listed as the first axiom of~${\czf{}}$:
\begin{equation*}
	\forall z\ (z \in x \lor P \leftrightarrow z \in y \lor P) \to x = y \lor P\period
\end{equation*}
This formula cannot be derived for arbitrary~${P}$.\footnote{\label{ft:con_zfc}Let~${P}$ be the (bounded) formula stating the consistency of~${\zfc{}}$. Let~${x}$ be empty and~${y}$ nonempty precisely if~${P}$ holds. Then, the premise can easily be derived. However, by the disjunction property (cf.~\cite[Theorem~1.2]{Rathjen05}),~${\czf{}}$ now either proves or disproves the consistency of~${\zfc{}}$, which cannot be.} In fact, Friedman himself studied a variant of his translation for an intuitionistic set theory but had to abandon extensionality (as well as the full powerset axiom) for similar reasons (see the definition of the system~${Q}$ at \cite[p.~25]{Friedman78}).

The problem is that the disjunctions in the premise and in the conclusion involve distinct \emph{levels} of the sets~${x}$ and~${y}$: In the premise,~${P}$ is added to statements concerning \emph{elements} of our sets. In the conclusion,~${P}$ is added to statements regarding the objects~${x}$ and~${y}$ \emph{themselves}. What we need is a translation that does not add~${P}$ ``outside'' of~${x = y}$ but, instead, incorporates it into the \emph{definition} of equality itself so that the addition of~${P}$ still remains on the level of elements.

Consider the following recursive definition of predicates~${='}$ and~${\in'}$, which very much resembles how equality and set membership are handled in realizability interpretations of set theory (see, e.g., the first few lines of \cite[Definition~2.2]{McCarty86}). Notice that only the definition of~${\in'}$ refers to the original predicate~${\in}$:
\begin{align*}
	x =' y &:\Longleftrightarrow \forall z.\ (z \in' x \to z \in' y) \land (z \in' y \to z \in' x)\comma\\
	x \in' y &:\Longleftrightarrow \exists z \in y.\ x =' z\period
	\intertext{With the help of set induction, it is possible to show that~${='}$ and~${\in'}$ behave precisely as~${=}$ and~${\in}$, respectively. Hence, we can use this description as a starting point for our definition. For our translation, we will introduce~${P}$ into both predicates in a hereditary fashion. In fact, consider the following slight variation:}
	x =_P y &:\Longleftrightarrow \forall z.\ (z \in_P x \to z \in_P y) \land (z \in_P y \to z \in_P x)\comma\\
	x \in_P y &:\Longleftrightarrow P \lor \exists z \in y.\ x =_P z\period
\end{align*}
Given~${\varphi}$, we call the formula~${\varphi_P}$ that results from~${\varphi}$ by replacing all prime formulas~${x = y}$,~${x \in y}$, and~${\bot}$ with~${x =_P y}$,~${x \in_P y}$, and~${P}$, respectively, the \emph{hereditary Friedman-Dragalin translation} of~${\varphi}$ with respect to~${P}$ (see Definition~\ref{def:varphi_P}).
As communicated to us by Michael Rathjen, a hereditary translation in our sense has independently been studied by him and Takako Nemoto.

In Theorem~\ref{thm:translation}, we will see that whenever~${\varphi}$ can be derived in~${\czf{}}$ with full separation or in~${\izf}$, then~${\varphi_P}$ can also be proved in the same system. Finally, in Theorem~\ref{thm:main}, we will use this result to show that independence of premise in the sense of Definition~\ref{def:IP} is admissible over both systems, answering \cite[Problem~1.5]{FNR23} by Frittaion, Nemoto, and Rathjen.

In Section~\ref{sec:prelims}, we begin by recalling the axioms of~${\czf{}}$ and~${\izf{}}$. In Section~\ref{sec:main}, we show how~${=_P}$ and~${\in_P}$ may be constructed in~${\czf{}}$, prove the admissibility of our hereditary translation and derive the independence of premise rule in the sense of Definition~\ref{def:IP}. Finally, in Section~\ref{sec:discuss}, we discuss why our method does not seem to work in a setting with only bounded separation.

\section{Preliminaries}\label{sec:prelims}

We are working in intuitionistic first-order logic with equality featuring the binary predicate symbol~${\in}$. We mostly use Greek symbols~${\varphi}$,~${\psi}$,~${\chi}$ for formulas although we also write~${P}$ for our translation parameter. We use both lower- and uppercase Latin symbols (except for~${P}$) as variables. Our logical system includes a symbol~${\bot}$ for falsity and negation~${\lnot\varphi}$ is read as an abbreviation for~${\varphi \to \bot}$. We also write~${x \notin y}$ in order to express~${\lnot x \in y}$. For readability, we often use dots after quantifiers instead of parentheses around the rest of the respective formula.

We write~${\forall x \in y\ \varphi}$ and~${\exists x \in y\ \varphi}$ in order to abbreviate both~${\forall x.\ x \in y \to \varphi}$ and~${\exists x.\ x \in y \land \varphi}$, respectively. We call them \emph{bounded quantifiers} and formulas in which each occurrence of a quantifier is bounded are called \emph{bounded formulas}.

We continue with the additional axioms and schemes required for defining \czf{} and \izf{}. We are using definitions from \cite[Section~2]{FNR23} in order to stay close to that article but also because these formulations seem to behave rather nicely with our arguments.
\begin{itemize}
	\item Extensionality: $\forall z\ (z \in x \leftrightarrow z \in y) \to x = y$
	\item Pairing: $\exists z.\ x \in z \land y \in z$
	\item Union: $\exists y \forall z.\ \exists w\ (z \in w \land w \in x) \to z \in y$
	\item Infinity: $\exists w \forall x.\ x \in w \leftrightarrow \forall y\ y \notin x \lor \exists y \in w\ \forall z.\ z \in x \leftrightarrow z \in y \lor z = y$
	\item Set induction: $\forall x\ (\forall y \in x\ \varphi(y) \to \varphi(x)) \to \forall x\ \varphi(x)$ for all formulas~${\varphi(x)}$
	\item Bounded separation: $\exists y \forall z.\ z \in y \leftrightarrow z \in x \land \varphi(z)$ for all bounded~${\varphi(z)}$
	\item Strong collection: $\forall y \in x\ \exists z\ \varphi(y, z) \to \exists w.$\\
	\phantom{}\hfill$\forall y \in x\ \exists z \in w\ \varphi(y, z) \land \forall z \in w\ \exists y \in x\ \varphi(y, z)$
	\item Subset collection: $\exists z \forall u.\ \forall v \in x\ \exists w \in y\ \varphi(v, w, u) \to \exists y' \in z.$\\
	\phantom{}\hfill$\forall v \in x\ \exists w \in y'\ \varphi(v, w, u) \land \forall w \in y'\ \exists v \in x\ \varphi(v, w, u)$
	\item Full separation: $\exists y \forall z.\ z \in y \leftrightarrow z \in x \land \varphi(z)$
	\item Powerset: $\exists y \forall z.\ z \subseteq x \to z \in y$
\end{itemize}
The system \czf{} consists of all these axioms except for the last two: full separation and the powerset axiom. The system \izf{} consists of all axioms (although bounded separation and subset collection are subsumed by full separation and the powerset axiom, respectively).

\section{A hereditary Friedman-Dragalin translation}\label{sec:main}

Given a sentence\footnote{Many of the following results can also be derived for \emph{formulas}~${P}$. However, Corollary~\ref{cor:repl_ex_falso} becomes false if extended to arbitrary formulas~${P}$. Thus, our main argument in Theorem~\ref{thm:main} is designed to only require results concerning \emph{sentences}~${P}$.}~${P}$, let us recall the intended properties of both~${=_P}$ and~${\in_P}$ from the introduction:
\begin{align*}
	x =_P y &\Longleftrightarrow \forall z.\ (z \in_P x \to z \in_P y) \land (z \in_P y \to z \in_P x)\comma\\
	x \in_P y &\Longleftrightarrow P \lor \exists u \in y.\ x =_P u\period
\end{align*}
Although such definitions are routine in \czf{} (cf.~\cite[Section~9.3]{AR10} and~\cite[Section~2.4]{Rathjen23}), we make the construction explicit in order to stay more transparent and self-contained. First, we focus on the definition of~${=_P}$:
\begin{definition}
	For any sentence~${P}$, we define
	\begin{align*}
		x =_P y :\equiv \exists A.\ (x, y) \in A \land \forall (u, v) \in A.\ &\forall a \in u\ (P \lor \exists b \in v\ (a, b) \in A) \land{}\\
		&\forall b \in v\ (P \lor \exists a \in u\ (a, b) \in A)\period
	\end{align*}
\end{definition}
Notice that in the relevant conjunction of our formula,~${A}$ only appears positively. Compare this to the concept of \emph{inductive definitions} (cf.~\cite[Section~12.1]{AR10} and~\cite[Section~2.4]{Rathjen23}).
\begin{lemma}\label{lem:equiv_rel}
	For any sentence~${P}$, the system \czf{} proves that~${=_P}$ is an equivalence relation.
\end{lemma}
\begin{proof}
	We begin with reflexivity. Let~${x}$ be arbitrary and let~${A}$ consist of precisely those pairs~${(u, u)}$ for which~${u}$ lies in the transitive closure of~${\{x\}}$. (The set~${A}$ can be constructed using strong collection.) Immediately, this entails~${(x, x) \in A}$. Moreover, consider a pair~${(u, u) \in A}$ and let~${a \in u}$ be arbitrary. By assumption,~${u}$ lies in the transitive closure of~${\{x\}}$, which must, therefore, also hold for~${a}$. Hence, we choose~${b := a \in u}$, which satisfies~${(a, b) \in A}$. We conclude~${x =_P x}$.
	
	Let us continue with symmetry. Assume that~${A}$ witnesses~${x =_P y}$. Using strong collection, we may define~${B}$ to be the set consisting of precisely those pairs~${(v, u)}$ with~${(u, v) \in A}$. Clearly,~${(x, y) \in A}$ entails~${(y, x) \in B}$. Now, assume~${(v, u) \in B}$, which implies~${(u, v) \in A}$. Let~${b \in v}$ be arbitrary, then either~${P}$ holds, or there is some~${a \in u}$ with~${(a, b) \in A}$. Consequently, either~${P}$ holds or there is some~${a \in u}$ with~${(b, a) \in B}$. The other half of the conjunction works analogously.
	
	Finally, we prove transitivity. Assume that~${x =_P y}$ and~${y =_P z}$ hold, witnessed by sets~${A}$ and~${B}$, respectively. Let~${C}$ be the set consisting precisely of those pairs~${(u, w)}$ for which there exists~${v}$ with~${(u, v) \in A}$ and~${(v, w) \in B}$. In order to construct~${C}$, take the product of~${A}$ and~${B}$, and consider the singleton~${\{(u, w)\}}$ for each~${((u, v_a), (v_b, w)) \in A \times B}$. With the help of bounded separation, we may even consider the subsingleton~${\{(u, w) \mid v_a = v_b\}}$. Using strong collection, we construct the family of all such sets, whose union results in~${C}$.
	
	From~${(x, y) \in A}$ and~${(y, z) \in B}$, it is clear that~${(x, z) \in C}$ holds. Let~${(u, w) \in C}$. By definition, there is~${v}$ with~${(u, v) \in A}$ and~${(v, w) \in B}$. From the former together with the definition of~${A}$, we know that for each~${a \in u}$, we have~${P}$ or there is~${b \in v}$ with~${(a, b) \in A}$. If~${P}$ holds, we are done. Otherwise, if~${b \in v}$ with~${(a, b) \in A}$ exists, we employ both~${(v, w) \in B}$ and the definition of~${B}$: Here, we know that for our~${b \in v}$, we have~${P}$ or there is~${c \in w}$ with~${(b, c) \in B}$. If~${P}$ holds, we are done as before. If~${c \in w}$ with~${(b, c) \in B}$ exists, we conclude that for our~${a \in u}$ from the beginning, there is~${c \in w}$ with~${(a, c) \in C}$, by definition of~${C}$. The other direction, i.e.~the statement~${\forall c \in w.\ P \lor \exists a \in u\ (a, c) \in C}$ is proved analogously.
\end{proof}

\begin{lemma}\label{lem:ex_falso_etc}
	For any sentences~${P}$ and~${Q}$, the system \czf{} proves
	\begin{enumerate}[label=\alph*)]
		\item~${P \to x =_P y}$,
		\item~${(P \to Q) \to x =_P y \to x =_Q y}$,
		\item~${x =_\bot y \leftrightarrow x = y}$,
	\end{enumerate}
	for all sets~${x}$ and~${y}$.
\end{lemma}
Together, b) and c) lead to the implication~${\lnot P \to x =_P y \to x = y}$.
\begin{proof}
	For a), assume that~${P}$ holds. We claim that already~${A := \{(x, y)\}}$ witnesses the claimed equality: First,~${(x, y) \in A}$ is satisfied. Moreover, for any~${(u, v) \in A}$ and~${a \in u}$, we have~${P}$. Similarly for any~${b \in v}$.
	
	Item b) follows from the fact that~${P}$ only occurs positively in the formula~${x =_P y}$.
	For c), assume that~${x =_\bot y}$ holds, witnessed by a set~${A}$. We claim that~${(u, v) \in A}$ entails~${u = v}$. Assume that the implication~${(a, b) \in A \to a = b}$ has already been established for all sets~${a}$ and~${b}$ with~${a \in u}$. Then, for any~${a \in u}$, there is~${b \in v}$ with~${(a, b) \in A}$. By assumption,~${a \in u}$ and~${(a, b) \in A}$ lead to~${a = b}$. Hence,~${a = b \in v}$ holds. Conversely, for any~${b \in v}$, there is~${a \in u}$ with~${(a, b) \in A}$, and the same argument implies~${b = a \in u}$. Thus,~${u = v}$ follows.
	By set induction along~${u}$, we conclude that~${(u, v) \in A}$ entails~${u = v}$. Together with~${(x, y) \in A}$, we arrive at our original claim.
	The direction from~${x = y}$ to~${x =_\bot y}$ is trivial.
\end{proof}
With equality in place, the definition of set membership follows easily. In Proposition~\ref{prop}, we will see that our definition is ``correct'' in the sense that it satisfies extensionality with respect to~${=_P}$.
\begin{definition}\label{def:x_inP_y}
	For any sentence~${P}$, we define
	\begin{equation*}
		x \in_P y :\equiv P \lor \exists u \in y.\ x =_P u\period
	\end{equation*}
\end{definition}
\begin{lemma}\label{lem:properties_in}
	For any sentence~${P}$, the system \czf{} proves
	\begin{enumerate}[label=\alph*)]
		\item~${x \in y \to x \in_P y}$,
		\item~${x =_P y \to y \in_P z \to x \in_P z}$,
		\item~${x \in_P y \to y =_P z \to x \in_P z}$,
	\end{enumerate}
	for all sets~${x}$,~${y}$, and~${z}$.
\end{lemma}
\begin{proof}
	Item a) follows from the reflexivity of~${=_P}$.
	For b), assume~${x =_P y}$ and~${y \in_P z}$. If~${P}$ holds, we are done. If there is~${u \in z}$ with~${y =_P u}$, then also~${x =_P u}$ follows by transitivity. We arrive at~${x \in_P z}$.
	
	For c), assume~${x \in_P y}$ and~${y =_P z}$. If~${P}$ holds, we are done. Assume that there is~${u \in y}$ with~${x =_P u}$. Let~${A}$ witness~${y =_P z}$. Then,~${u \in y}$ leads to~${P}$ (and we are done, as before) or there is~${v \in z}$ with~${(u, v) \in A}$. In particular,~${A}$ is now also a witness for~${u =_P v}$. By transitivity, we derive~${x =_P v}$. We conclude that there is~${v \in z}$ with~${x =_P v}$. Thus, our claim~${x \in_P z}$ is satisfied.
\end{proof}

\begin{proposition}\label{prop}
	For any sentence~${P}$, the system \czf{} proves
	\begin{equation*}
		x =_P y \longleftrightarrow \forall z.\ (z \in_P x \to z \in_P y) \land (z \in_P y \to z \in_P x)\period
	\end{equation*}
\end{proposition}
\begin{proof}
	The direction from left to right follows from Lemmas~\ref{lem:equiv_rel} and~\ref{lem:properties_in}. For the other direction, assume that the right hand side holds. For any~${z \in x}$, we have~${z \in_P x}$ by Lemma~\ref{lem:properties_in} and, hence,~${z \in_P y}$ by assumption. In this case, either~${P}$ holds or there exist sets~${A}$ and~${u \in y}$ such that~${A}$ witnesses~${z =_P u}$. This can be reformulated to the following statement: For any~${z \in x}$, there exists a set~${A}$ such that either~${A = \emptyset}$ and~${P}$ holds or there is~${u \in y}$ such that~${A}$ witnesses~${z =_P u}$.
	
	Now, we apply strong collection, which yields a set~${B}$ such that for any~${z \in x}$ an~${A}$ as above can be found in~${B}$ and, conversely, for any~${A \in B}$ there is some~${z \in x}$ satisfying the statement. Let~${C}$ be the union of~${B}$. Furthermore, let~${C'}$ be the set that results from a similar construction applied to~``${z \in_P y \to z \in_P x}$'': For any~${z \in y}$, there exists a set~${A' \in B'}$ such that either~${A' = \emptyset}$ and~${P}$ holds or there is~${u \in x}$ such that~${A'}$ witnesses~${u =_P z}$. (Here, we applied symmetry from Lemma~\ref{lem:equiv_rel} to the equality~${z =_P u}$.) Moreover,~${B'}$ only consists of~${A'}$ for which a~${z \in y}$ with this property exists. Then, define~${C'}$ to be the union of~${B'}$.	
	Finally, we define~${D := \{(x, y)\} \cup C \cup C'}$ and claim that this set witnesses~${x =_P y}$.
	
	Trivially,~${(x, y) \in D}$ holds. Now, let~${(u, v) \in D}$ be arbitrary. There are three cases:~${(u, v) = (x, y)}$, or one of~${(u, v) \in C}$ or~${(u, v) \in C'}$. In the first case, the following holds by definition of~${B}$: For any~${a \in u = x}$, there is~${A \in B}$ with either~${A = \emptyset}$ and~${P}$ or there is~${b \in v = y}$ such that~${A}$ witnesses~${a =_P b}$. In particular, for any~${a \in u}$, (there is~${A \in B}$ such that) we have~${P}$ or there is~${b \in v}$ with~${(a, b) \in A \subseteq C \subseteq D}$. Similarly, for any~${b \in v = y}$, (there is~${A' \in B'}$ such that) we have~${P}$ or there is~${a \in u = x}$ with~${(a, b) \in A' \subseteq C' \subseteq D}$. 
	
	In the second case,~${(u, v) \in C}$ holds. Here, there is~${A \subseteq C}$ with~${(u, v) \in A}$ such that~${A}$ is witness of an equality~${=_P}$. Hence, for all~${a \in u}$, we have~${P}$ or some~${b \in v}$ with~${(a, b) \in A \subseteq C \subseteq D}$. Conversely, for all~${b \in v}$, we have~${P}$ or some~${a \in u}$ with~${(a, b) \in A \subseteq C \subseteq D}$. The situation is similar for~${(u, v) \in C'}$ as this set is also a union of witnesses of~${=_P}$.
\end{proof}
The argument reveals that the claim of Proposition~\ref{prop} still holds if we replace~${z \in_P x}$ and~${z \in_P y}$ in the premises (but, importantly, not the conclusions) with~${z \in x}$ and~${z \in y}$, respectively.
\begin{definition}\label{def:varphi_P}
	Given a formula~${\varphi}$ and a sentence~${P}$, we write~${\varphi_P}$ for the formula that results from~${\varphi}$ if all occurrences of~${=}$,~${\in}$, and~${\bot}$, respectively, are replaced with~${=_P}$,~${\in_P}$, and~${P}$.
\end{definition}

\begin{corollary}\label{cor:repl_ex_falso}
	Given formulas~${\varphi(x)}$ and~${\psi}$ together with a sentence~${P}$, we derive the following:
	\begin{enumerate}[label=\alph*)]
		\item~${x =_P y \to \varphi_P(x) \to \varphi_P(y)}$,
		\item~${P \to \psi_P}$,
		\item~${\lnot P \to (\psi_P \leftrightarrow \psi)}$,
	\end{enumerate}
	for all~${x}$ and~${y}$.\footnote{Let us make explicit that neither~${x}$ nor~${y}$ is bounded in~${\varphi(x)}$.}
\end{corollary}
Before we continue with the proof, let us emphasize that b) becomes false if extended to (unrestricted) \emph{formulas}~${P}$: Consider~${P :\equiv x = \emptyset}$ and~${\psi :\equiv \forall x \bot}$. Then,~${P \to \psi_P}$ abbreviates the non-derivable statement~${x = \emptyset \to \forall x\ x = \emptyset}$.
\begin{proof}[Proof of Corollary~\ref{cor:repl_ex_falso}]
	For a), we use Lemmas~\ref{lem:equiv_rel} and~\ref{lem:properties_in} to derive the claim for prime formulas. The result for arbitrary formulas follows easily by induction along their structure. In the case of implications, we apply the symmetric statement (assuming~${y =_P x}$ instead of~${x =_P y}$) to our premise. Notice that the cases involving quantifiers are unproblematic with respect to our assumption~${x =_P y}$ since neither~${x}$ nor~${y}$ is bound by~${\varphi(x)}$ and~${P}$ has no free variables.
	
	For b), we use Lemma~\ref{lem:ex_falso_etc}~a) to derive the claim for~${x = y}$. For prime formulas~${x \in y}$, it follows trivially from Definition~\ref{def:x_inP_y}. Finally, for~${\bot}$, it is a simple tautology. Similar to before, we apply induction along the structure to cover all other formulas. In the case of implications, we only need the derive the conclusion and then apply weakening. Again, we can introduce universal quantifiers since~${P}$ has no free variables.
	
	For c), we use Lemma~\ref{lem:ex_falso_etc}~b) and~c) as well as Definition~\ref{def:x_inP_y} and extensionality to derive the claim for prime formulas. In order to extend this to non-prime formulas, we argue as in a). In the case of implications, we apply the respective opposite directions to our premise.
\end{proof}
We continue with the last preparatory step for our main result: the admissibility of our translation.
\begin{theorem}\label{thm:translation}
	Let~${T}$ be either~${\czf{}}$ together with full separation or~${\izf{}}$. Let~${P}$ be a sentence and~${\varphi}$ a formula. If~${T}$ proves~${\varphi}$, then it also proves~${\varphi_P}$.
\end{theorem}
\begin{proof}
	Inductively, we translate a proof of~${\varphi}$ to one of~${\varphi_P}$. All axioms and rules that do not explicitly involve certain predicates can be transferred directly. Importantly, notice that the lack of free variables in~${P}$ allows us to apply~${\forall}$-introductions and~${\exists}$-eliminations as in the original proof. Finally, we discuss all remaining axioms, i.e.\ those involving equality and falsum as well as all axioms of \czf{} (except for bounded separation) and \izf{}.
	
	\textbf{Equality}: In Lemma~\ref{lem:equiv_rel}, we have seen that~${=_P}$ is an equivalence relation. Moreover, it satisfies the equality axioms with respect to our predicate~${\in_P}$ by Lemma~\ref{lem:properties_in}.
	
	\textbf{Ex falso quodlibet}: This axiom is satisfied by Corollary~\ref{cor:repl_ex_falso}.
	
	\textbf{Extensionality}: This is covered by Proposition~\ref{prop}.
	
	\textbf{Pairing}: Given~${x}$ and~${y}$, let~${z}$ be a set satisfying both~${x \in z}$ and~${y \in z}$. By Lemma~\ref{lem:properties_in}, this set~${z}$ also satisfies both~${x \in_P z}$ and~${y \in_P z}$.
	
	\textbf{Union}: Given~${x}$, let~${y}$ be a witness of the union axiom, i.e.~satisfying~${\bigcup x \subseteq y}$. Let~${z}$ and~${w}$ be such that~${w \in_P x}$ and~${z \in_P w}$ hold. If~${w \in_P x}$ is satisfied because of~${P}$, then~${z \in_P y}$ follows trivially. Otherwise, we find~${w' \in x}$ with~${w =_P w'}$. From the latter and~${z \in_P w}$, we derive~${z \in_P w'}$ by Lemma~\ref{lem:properties_in}.
	 
	If this relation~${z \in_P w'}$ holds, just like before because of~${P}$, then~${z \in_P y}$, again, follows trivially. Otherwise, we find~${z' \in w'}$ with~${z =_P z'}$. From~${z' \in w'}$ and~${w' \in x}$, we conclude~${z' \in y}$. This leads to~${z' \in_P y}$ by Lemma~\ref{lem:properties_in}. A second application of this lemma together with~${z =_P z'}$ yields~${z \in_P y}$.
	
	\textbf{Infinity}: Let~${\omega}$ be a witness of the infinity axiom. For the direction from left to right, consider~${x}$ with~${x \in_P \omega}$. If this relation is satisfied because of~${P}$, our claim follows from Corollary~\ref{cor:repl_ex_falso}. Otherwise, there is~${x' \in \omega}$ with~${x =_P x'}$. By assumption on~${\omega}$, there are two cases: In the first case,~${x'}$ is empty. Assume~${y \in_P x}$. Then, also~${y \in_P x'}$ holds by Lemma~\ref{lem:properties_in}. Since~${x'}$ is empty, we conclude~${P}$. In the second case, there is~${y \in \omega}$ with~${x' = y \cup \{y\}}$. Assume~${z \in_P x}$. From Lemma~\ref{lem:properties_in}, we derive~${z \in_P x'}$. If this relation holds because of~${P}$, the right hand side~${z \in_P y \lor z =_P y}$ follows trivially. Otherwise, we find~${z' \in x'}$ with~${z =_P z'}$. The former entails~${z' \in y}$ or~${z' = y}$. Hence, by Lemma~\ref{lem:properties_in} and extensionality (applied to~${z' =_P z'}$), we arrive at~${z' \in_P y}$ or~${z' =_P y}$, respectively. Finally, Lemmas~\ref{lem:equiv_rel} and \ref{lem:properties_in} with~${z =_P z'}$ lead to our claim~${z \in_P y \lor z =_P y}$. For the converse direction, assume that this disjunction is satisfied. Beginning with~${z \in_P y}$, our claim~${z \in_P x}$ follows trivially  if~${P}$ holds. Otherwise, there exists~${z' \in y}$ with~${z =_P z'}$. By definition of~${x'}$, we have~${z' \in x'}$. Lemma~\ref{lem:properties_in} leads to~${z' \in_P x'}$ and Lemmas~\ref{lem:equiv_rel} and~\ref{lem:properties_in} together with~${x =_P x'}$ and~${z =_P z'}$ yield our claim~${z \in_P x}$. In the other case with~${z =_P y}$, we apply Lemma~\ref{lem:properties_in} to the fact~${y \in x'}$ and use a similar argument to derive~${z \in_P x}$.
	
	For the direction from right to left, let us first assume that~${x}$ is a set entailing~${P}$ for all~${y \in_P x}$. Clearly, this yields~${x =_P \emptyset}$ (for~${y \in \emptyset}$, we get~${y \in_P x}$ by ex falso quodlibet; for~${y \in_P x}$, we get~${P}$ and, therefore,~${y \in_P \emptyset}$ by Lemma~\ref{lem:ex_falso_etc}). Together with~${\emptyset \in \omega}$ as well as Lemma~\ref{lem:properties_in}, we arrive at~${x \in_P \omega}$.
	
	Finally, let us assume that~${x}$ is a set for which an element~${y \in_P \omega}$ exists that satisfies~${\forall z.\ z \in_P x \leftrightarrow z \in_P y \lor z =_P y}$. W.l.o.g., we can assume~${y \in \omega}$: If~${y \in_P \omega}$ holds because of~${P}$, we are done. Otherwise, if there is~${y' \in \omega}$ with~${y =_P y'}$, then~${y}$ can be replaced with~${y'}$ in our assumption using Corollary~\ref{cor:repl_ex_falso}. We aim to prove~${x =_P y \cup \{y\}}$. If~${z \in_P x}$ holds, we have~${z \in_P y}$ or~${z =_P y}$. If the former case holds because of~${P}$, we are done. If there is~${z' \in y}$ with~${z =_P z'}$, then also~${z' \in_P y \cup \{y\}}$ holds by Lemma~\ref{lem:properties_in}. We conclude~${z \in_P y \cup \{y\}}$ using the same result. In the latter case with~${z =_P y}$, we use~${y \in y \cup \{y\}}$ together with the same lemma. Now, assume that~${z \in y \cup \{y\}}$ holds. If~${z \in y}$, then~${z \in_P y}$ follows by Lemma~\ref{lem:properties_in} and we derive~${z \in_P x}$ by assumption. If~${z = y}$, then~${z =_P y}$ follows by extensionality and the reflexivity of~${=_P}$. Again, our assumption leads to~${z \in_P x}$.
	By definition of~${\omega}$ together with~${y \in \omega}$, we have~${y \cup \{y\} \in \omega}$. By Lemma~\ref{lem:properties_in} and~${x =_P y \cup \{y\}}$, we conclude~${x \in_P \omega}$.

	\textbf{Set induction}: Consider the instance of this scheme with respect to a formula~${\varphi(x)}$. In an effort to derive its translation, assume~${\forall y \in_P x\ \varphi_P(y) \to \varphi_P(x)}$ for all sets~${x}$. In order to derive~${\forall z \in x\ \varphi_P(z) \to \varphi_P(x)}$ for all~${x}$, assume that~${\varphi_P(z)}$ holds for all~${z \in x}$. Consider an arbitrary~${y \in_P x}$. In the first case, we have~${P}$, which immediately leads to~${\varphi_P(y)}$. In the second one, we find~${y' \in x}$ with~${y =_P y'}$. By assumption, we have~${\varphi_P(y')}$. Together with Lemma~\ref{lem:equiv_rel} and Corollary~\ref{cor:repl_ex_falso}, this entails~${\varphi_P(y)}$. In conclusion, we have~${\forall y \in_P x\ \varphi_P(y)}$. By our assumption, this yields~${\varphi_P(x)}$. Since~${x}$ was arbitrary, the premise of set induction with respect to the formula~${\varphi_P(x)}$ is satisfied. Finally, we conclude that~${\varphi_P(x)}$ holds for all sets~${x}$.
	
	\textbf{Full separation}: Consider the instance of this scheme with respect to a formula~${\varphi(z)}$. Let~${y}$ be the set that results from applying (full) separation to~${x}$ using~${\varphi(z)}$. We claim that~${y}$ also witnesses the translated instance. First, assume~${z \in_P y}$. Then, if~${P}$ holds, both~${z \in_P x}$ and~${\varphi_P(z)}$ follow immediately by Definition~\ref{def:x_inP_y} and Corollary~\ref{cor:repl_ex_falso}. Otherwise, we have~${z' \in y}$ with~${z =_P z'}$. By assumption on~${y}$, both~${z' \in x}$ and~${\varphi_P(z')}$ follow. Applying Lemma~\ref{lem:properties_in} to the former as well as Lemma~\ref{lem:equiv_rel} and Corollary~\ref{cor:repl_ex_falso} to the latter, we arrive at both~${z \in_P x}$ and~${\varphi_P(z)}$.
	
	For the other direction, assume~${z \in_P x}$ and~${\varphi_P(z)}$. If the former holds because of~${P}$, then~${z \in_P y}$ holds by Lemma~\ref{lem:ex_falso_etc}. Otherwise, we have~${z' \in x}$ with~${z =_P z'}$. Now, Corollary~\ref{cor:repl_ex_falso} entails~${\varphi_P(z')}$ and, by definition of~${y}$, we conclude~${z' \in y}$. Finally, by Lemma~\ref{lem:properties_in}, we derive~${z \in_P y}$.
	
	\textbf{Strong collection}: Consider the instance of this scheme with respect to a formula~${\varphi(y, z)}$. Assume that for all sets~${y}$ with~${y \in_P x}$, there is~${z}$ with~${\varphi_P(y, z)}$. By Lemma~\ref{lem:properties_in}, this entails that for any~${y \in x}$, there is~${z}$ with~${\varphi_P(y, z)}$. Now, let~${w}$ be the set that results from applying the principle of strong collection to~${\varphi_P(y, z)}$ under this assumption. In order to derive~${\forall y \in_P x\ \exists z \in_P w\ \varphi_P(y, z)}$, consider an arbitrary set~${y}$ with~${y \in_P x}$. If this relation is satisfied due to~${P}$, then~${\exists z \in_P w\ \varphi_P(y, z)}$ holds trivially. Otherwise, we find~${y' \in x}$ with~${y =_P y'}$. By assumption on~${w}$, there is~${z \in w}$ with~${\varphi_P(y', z)}$. By Lemma~\ref{lem:properties_in}, we have~${z \in_P w}$ and Corollary~\ref{cor:repl_ex_falso} allows us to derive~${\varphi_P(y, z)}$. The argument for~${\forall z \in_P w\ \exists y \in_P x\ \varphi_P(y, z)}$ is analogous.
	
	\textbf{Subset collection}: Consider the instance of this scheme with respect to a formula~${\varphi(v, w, u)}$. We claim that its translation is derivable. For this, let~${x}$ and~${y}$ be arbitrary sets. Now, consider an instance of subset collection with respect to the formula
	\begin{equation*}
		\bar{\varphi}(v, \bar{w}, u) := (\bar{w} = \emptyset \land P) \lor \exists w' (\bar{w} = \{w'\} \land \varphi_P(v, w', u))\period
	\end{equation*}
	Let~${\bar{z}}$ be the set that results from this principle using~${\bar{\varphi}(v, w, u)}$ together with~${x}$ and~${\bar{y} := \{\emptyset\} \cup \{\{a\} \mid a \in y\}}$. The latter set can easily be constructed from~${y}$ using pairing (to define~${\{\emptyset\}}$) as well as strong collection (to define~${\{\{a\} \mid a \in y\}}$) and union. We claim that~${z := \{\bigcup y' \mid y' \in \bar{z}\}}$ (defined using union and strong collection) witnesses the translated principle for our original formula.
	
	Let~${u}$ be arbitrary and assume~${\forall v \in_P x\ \exists w \in_P y\ \varphi_P(v, w, u)}$. Let~${v \in x}$ be arbitrary. By Lemma~\ref{lem:properties_in}, we have~${v \in_P x}$ and, therefore, find~${w}$ with~${w \in_P y}$ and~${\varphi_P(v, w, u)}$. If~${w \in_P y}$ holds via~${P}$, then there is~${\bar{w} := \emptyset \in \bar{y}}$ with~${\bar{\varphi}(v, \bar{w}, u)}$. Otherwise, there is~${w' \in y}$ with~${w =_P w'}$. From Corollary~\ref{cor:repl_ex_falso}, we derive~${\varphi_P(v, w', u)}$. Hence, there is~${\bar{w} := \{w'\} \in \bar{y}}$ with~${\bar{\varphi}(v, \bar{w}, u)}$. In conclusion, for any~${v \in x}$, there is~${\bar{w} \in \bar{y}}$ with~${\bar{\varphi}(v, \bar{w}, u)}$. Thus, there is~${\bar{y}' \in \bar{z}}$ satisfying~${\forall v \in x\ \exists \bar{w} \in \bar{y}'\ \bar{\varphi}(v, \bar{w}, u)}$ and~${\forall \bar{w} \in \bar{y}'\ \exists v \in x\ \bar{\varphi}(v, \bar{w}, u)}$. We set~${y' := \bigcup{\bar{y}'} \in z}$.
	
	Let~${v}$ be a set with~${v \in_P x}$. If this holds because of~${P}$, then~${\exists w \in_P y'\ \varphi_P(v, w, u)}$ holds by Corollary~\ref{cor:repl_ex_falso}. Otherwise, there is~${v' \in x}$ with~${v =_P v'}$. By assumption on~${\bar{y}'}$, we find~${\bar{w} \in \bar{y}'}$ with~${\bar{\varphi}(v', \bar{w}, u)}$. If~${\bar{w} = \emptyset}$ holds, then the latter entails~${P}$, which yields~${\exists w \in_P y'\ \varphi_P(v, w, u)}$ as before. Otherwise,~${\bar{w}}$ is a singleton. Hence, there is~${w \in y'}$ (which is the union of~${\bar{y}'}$) satisfying~${\bar{w} = \{w\}}$. By Lemma~\ref{lem:properties_in}, the former yields~${w \in_P y'}$. Moreover,~${\bar{\varphi}(v', \bar{w}, u)}$ entails~${\varphi_P(v', w, u)}$ and, using Corollary~\ref{cor:repl_ex_falso}, also~${\varphi_P(v, w, u)}$ from~${v =_P v'}$. For the second conjunct, consider a set~${w}$ with~${w \in_P y'}$. Similar to before, if this relation holds because of~${P}$, then~${\exists v \in_P x\ \varphi_P(v, w, u)}$ is immediate. Otherwise, we find~${w' \in y'}$ with~${w =_P w'}$. By definition of~${y'}$, we have~${\bar{w} := \{w'\} \in \bar{y}'}$. Hence, there is~${v \in x}$ satisfying~${\bar{\varphi}(v, \bar{w}, u)}$. By Lemma~\ref{lem:properties_in}, the former yields~${v \in_P x}$. By definition, the latter entails~${\varphi_P(v, w', u)}$. Using Corollary~\ref{cor:repl_ex_falso} and~${w =_P w'}$, we get~${\varphi_P(v, w, u)}$.
	
	\textbf{Powerset}: Let~${y}$ be the powerset of some set~${x}$. Given any~${z}$, assume~${z \subseteq_P x}$, i.e.~that~${w \in_P z}$ entails~${w \in_P x}$ for arbitrary~${w}$. We define~${z' := \{w \in x \mid w \in_P z\}}$ using full separation\footnote{We require a set~${z'}$ that validates the \emph{translation} of being the intersection of~${x}$ and~${z}$. In Section~\ref{sec:discuss}, we will discuss why a straightforward definition~${z' := x \cap z}$ does not work in general.}. By definition, it is immediately clear that~${z' \in y}$ and, using Lemma~\ref{lem:properties_in}, also~${z' \in_P y}$ holds. In order to arrive at our claim, we prove~${z =_P z'}$: Let~${w \in_P z}$. Then,~${w \in_P x}$ follows by assumption. If this relation holds because of~${P}$, then~${w \in_P z'}$ follows from Lemma~\ref{lem:ex_falso_etc}. Otherwise, we find~${w' \in x}$ with~${w =_P w'}$. From~${w \in_P z}$, we then derive~${w' \in_P z}$ by Lemmas~\ref{lem:equiv_rel} and~\ref{lem:properties_in}. Hence, we have~${w' \in z'}$. Thus, Lemma~\ref{lem:properties_in} leads us to~${w \in_P z'}$. The other inclusion follows by definition of~${z'}$. Finally,~${z =_P z'}$ follows from Proposition~\ref{prop}. Together with~${z' \in_P y}$, Lemma~\ref{lem:properties_in} leads this to our claim~${z \in_P y}$.
	
	For the other direction, let~${z}$ be a set with~${z \in_P y}$. If this relation is satisfied because of~${P}$, then~${z \subseteq_P x}$ is trivial. Otherwise, we find~${z' \in y}$ with~${z =_P z'}$. Then,~${z' \subseteq x}$ must hold. Let~${w}$ be arbitrary with~${w \in_P z}$. By Lemma~\ref{lem:properties_in}, we have~${w \in_P z'}$. If this holds because of~${P}$, then our claim~${w \in_P x}$ is trivial. Otherwise, we find~${w' \in z'}$ with~${w =_P w'}$. The former yields~${w' \in x}$. Finally, together with~${w =_P w'}$, Lemma~\ref{lem:properties_in} leads to our claim~${w \in_P x}$.
\end{proof}
Finally, everything is prepared for the derivation of our main result:
\begin{theorem}\label{thm:main}
	Both systems \czf{} with full separation and \izf{} admit the independence of premise rule in the sense of Definition~\ref{def:IP}.
\end{theorem}
\begin{proof}
	Let~${T}$ be either \czf{} with full separation or \izf{}. Assume that~${T}$ proves the implication~${\varphi :\equiv \lnot \psi \to \exists x\ \chi(x)}$ where~${x}$ may not occur freely in~${\psi}$. W.l.o.g., we may assume that~${\psi}$ is a sentence by replacing each variable~${v}$ in~${\varphi}$ that occurs freely in~${\psi}$ by a fresh constant~${c_v}$, which we add to our language. Let~${P :\equiv \lnot\lnot\psi}$ and apply Theorem~\ref{thm:translation}, which yields a proof of
	\begin{equation*}
		\varphi_P \equiv (\psi_{\lnot\lnot\psi} \to \lnot\lnot\psi) \to \exists x\ \chi_{\lnot\lnot\psi}(x)
	\end{equation*}
	in~${T}$. Recalling the argument from the introduction regarding Heyting arithmetic, assume~${\lnot\psi}$. Applying Corollary~\ref{cor:repl_ex_falso} to~${\psi_{\lnot\lnot\psi}}$, we derive~${\psi}$. Together with~${\lnot\psi}$, this yields~${\bot}$. In conclusion, the implication~${\psi_{\lnot\lnot\psi} \to \lnot\lnot\psi}$ is provable and we arrive at a derivation of
	\begin{align*}
		T &\vdash \exists x\ \chi_{\lnot\lnot\psi}(x)\period
	\intertext{Now, we introduce the assumption~${\lnot\psi}$ after the existential quantifier. Using the same argument as before, in the presence of the conclusion~${\chi_{\lnot\lnot\psi}(x)}$, this assumption entails~${\chi(x)}$. If~${\psi}$ was a sentence from the beginning, we now have a derivation of}
		T &\vdash \exists x.\ \lnot\psi \to \chi(x)
	\end{align*}
	and have, hence, arrived at our claim. Otherwise, if~${\psi}$ contained free variables~${v}$, which we replaced with constants~${c_v}$, then this operation has to be reverted: In a final step, we replace each constant~${c_v}$ with its original variable~${v}$. For this operation to be sound, it is important that no replacing variable is captured by a quantifier. Precisely this property is ensured by our assumption that~${x}$ does not occur freely in~${\psi}$.
\end{proof}

\section{Issues regarding \czf{} with bounded separation}\label{sec:discuss}

In Theorem~\ref{thm:main}, we have seen that \czf{} with full separation admits the independence of premise rule from Definition~\ref{def:IP}. Moreover, in the introduction, we have claimed that our method does not seem to work for \czf{} itself. In the final paragraphs of this article, let us discuss where the issue lies. Since our proof of Theorem~\ref{thm:translation} reveals that the translations of all other axioms can be derived in \czf{}, we see that the problematic axiom must be the \emph{bounded separation} principle. In fact, as the fascinating result \cite[Corollary~9.5.7]{AR10} (see also \cite[Corollary~4.3]{Aczel13}) shows, the problem must already be visible when we try to build the intersection of two sets:\footnote{One has to be slightly careful here, since the cited results employ sets of axioms that are different to ours. This difference becomes visible in the absence of separation.}

Let~${x}$ and~${y}$ be arbitrary sets. We try to construct~${z}$ satisfying
\begin{equation*}
	w \in_P z\quad\longleftrightarrow\quad w \in_P x \land w \in_P y
\end{equation*}
for all~${w}$. A first attempt for a definition is given by intersection itself:~${z := x \cap y}$. However, if we try to prove the direction from right to left and have sets~${w}$,~${w_1}$, and~${w_2}$ (the latter two possibly distinct) with~${w =_P w_1}$ and~${w =_P w_2}$ as well as~${w_1 \in x}$ and~${w_2 \in y}$, how are we supposed to derive~${w \in_P z}$? To be more precise, let~${P}$ be a bounded sentence that is independent from \czf{} (such as its consistency). Define
\begin{equation*}
	x := \{\emptyset\} \quad \text{and} \quad y := \{\{\emptyset \mid P\}\}\period
\end{equation*}
Given~${z := x \cap y}$, the decidability of~${P}$ is equivalent to~${z}$ being inhabited with respect to~${\in_P}$: If~${P}$ holds, then~${w \in_P z}$ is satisfied for any~${w}$. If~${\lnot P}$ holds, then~${\emptyset}$ and~${\{\emptyset \mid P\}}$ are equal and, hence,~${\emptyset \in_P z}$ is satisfied. Conversely, if there is~${w \in_P z}$, then either~${P}$ follows or we find~${w' \in z}$ (with~${w =_P w'}$), which entails~${\emptyset = \{\emptyset \mid P\}}$ and, hence,~${\lnot P}$.

However, there is always a common set~${w}$ with both~${w \in_P x}$ and~${w \in_P y}$: We just choose~${w := \emptyset}$. Then,~${w \in_P x}$ is trivial and for~${w' := \{\emptyset \mid P\} \in y}$, we prove~${w =_P w'}$: If~${u \in w}$, then~${u \in_P w'}$ follows trivially since~${w}$ is empty. If~${u \in w'}$, then~${u = \emptyset}$ holds and~${P}$ is satisfied, which entails~${u \in_P w}$.

If~${z}$ was a witness for the translated intersection of~${x}$ and~${y}$, then both~${w \in_P z}$ and the conjunction of~${w \in_P x}$ and~${w \in_P y}$ must be equivalent. Since the latter is provable for~${w := \emptyset}$, the same would hold for the former and, by the disjunction property of \czf{}, either~${P}$ or~${\lnot P}$ would be provable.

We conclude that the translated intersection of~${x}$ and~${y}$ requires a different definition. A fruitful choice is given by
\begin{equation*}
	z := \{w \in x \mid w \in_P y\}\period
\end{equation*}
However, this definition requires full separation: Even if all quantifiers in~${P}$ are bounded, the statement~${w \in_P y}$ still contains an unbounded existential quantifier. This construction is, therefore, not available in \czf{}.

\subsection*{Acknowledgments}
I would like to thank Rosalie Iemhoff, Nicholas Pischke, Michael Rathjen, and Albert Visser for our correspondence as well as Takayuki Kihara and Lo\"{i}c Pujet for their interesting remarks after my talk at the Logic Colloquium 2026 in Swansea.

No parts of this article were written or worded by AI. Likewise, all definitions, ideas, and proofs were made without the use of AI (except for Footnote~\ref{ft:con_zfc}).

ChatGPT 5.6 Sol was employed for spell-checking as well as advanced literature research in Section~\ref{sec:intro}. (This way, we discovered that, in the setting of Heyting arithmetic, Albert Visser had already proved independence of premise using Friedman-Dragalin before us.) Finally, it provided an initial idea for Footnote~\ref{ft:con_zfc}.

\bibliographystyle{plain}
\bibliography{ip}
	
\end{document}